\documentclass{amsart}

\usepackage{amscd,amsfonts,amssymb,amsmath,amsthm,latexsym,booktabs,lscape,color,enumerate}
\usepackage{geometry,pb-diagram}
\usepackage{stmaryrd}
\usepackage{verbatim}
\usepackage{multicol}
\usepackage{geometry,pb-diagram}

\newtheorem{Theorem}{Theorem}[section]
\newtheorem{Proposition}[Theorem]{Proposition}
\newtheorem{Corollary}[Theorem]{Corollary}
\newtheorem{Lemma}[Theorem]{Lemma}
\newtheorem{Example}[Theorem]{Example}
\theoremstyle{definition}
\newtheorem{Remark}[Theorem]{Remark}
\newtheorem{Definition}[Theorem]{Definition}

\usepackage{tikz, ifthen}
\usetikzlibrary{calc}

\def\boldentry{\protect\unprotectedboldentry}

\newcommand{\tikztableau}[2][scale=0.6,every node/.style={font=\small}]{
    \def\newtableau{#2}
    \begin{array}{c}
    \begin{tikzpicture}[#1]
    \coordinate (x) at (-0.5,0.5);
    \coordinate (y) at (-0.5,0.5);
    \foreach \row in \newtableau {
        \coordinate (x) at ($(x)-(0,1)$);
        \coordinate (y) at (x);
        \foreach \entry in \row {
            \ifthenelse{\equal{\entry}{X}}
               {
                \node (y) at ($(y) + (1,0)$) {};
                \fill[color=gray!10] ($(y)-(0.5,0.5)$) rectangle +(1,1);
                \draw[color=gray] ($(y)-(0.5,0.5)$) rectangle +(1,1);
               }
               {
                \ifthenelse{\equal{\entry}{\boldentry X}}
                   {
                    \node (y) at ($(y) + (1,0)$) {};
                    \fill[color=gray] ($(y)-(0.5,0.5)$) rectangle +(1,1);
                    \draw ($(y)-(0.5,0.5)$) rectangle +(1,1);
                   }
                   {
                    \node (y) at ($(y) + (1,0)$) {\entry};
                    \draw ($(y)-(0.5,0.5)$) rectangle +(1,1);
                   }
               }
            }
        }
    \end{tikzpicture}
    \end{array}}
\newcommand{\tikztableausmall}[1]{\tikztableau[scale=0.45,every node/.style={font=\rm\small}]{#1}}

\def\sym{\operatorname{\mathsf{Sym}}}

\def\Qsym{\operatorname{\mathsf{QSym}}}
\def\Nsym{\operatorname{\mathsf{NSym}}}

\newcommand{\Id}{\mathrm{Id}}
\newcommand{\comp}{\mathrm{comp}}
\newcommand{\sgn}{\mathrm{sgn}}
\newcommand{\nn}{\mathrm{neg}}
\newcommand{\Z}{\mathrm{Z}}

\def\sort{\operatorname{sort}}
\def\tail{\operatorname{tail}}

\def \fS{{\mathfrak S}}
\def\ZZ{\mathbb{Z}}

\def\BB{\mathbb{B}}
\def\QQ{\mathbb{Q}}

\def\B{{\bf B}}
\def \HH{{H}}

\begin{document}

\title{The Pieri rule for dual immaculate quasi-symmetric functions}

\author[Berger\'on, S\'anchez-Ortega, Zabrocki]{Nantel Berger\'on$^{1,2}$, Juana S\'anchez-Ortega$^{1,2,3,4}$, Mike Zabrocki$^{1,2}$}

\address[1]{Fields Institute\\ Toronto, ON, Canada}
\address[2]{York University\\ Toronto, ON, Canada}
\address[3]{University of Toronto\\ Toronto, ON, Canada}
\address[4]{Universidad de M\'alaga\\ M\'alaga, Spain}

\email{bergeron@yorku.ca}
\email{jsanchezo@uma.es}
\email{zabrocki@mathstat.yorku.ca}



\maketitle

\begin{abstract}    
The immaculate basis of the non-commutative symmetric functions was recently introduced
by the first and third author to lift certain structures in the symmetric functions to the
dual Hopf algebras of the non-commutative and quasi-symmetric functions. 
It was shown that immaculate basis satisfies a positive, multiplicity free right Pieri rule.
It was conjectured that the left Pieri rule may contain signs but that it would be multiplicity free. 
Similarly, it was also conjectured that the dual quasi-symmetric basis would also
satisfy a signed multiplicity free Pieri rule.
We prove these two conjectures here.
\end{abstract}


\section{Introduction}
The algebra $\Qsym$ of quasi-symmetric functions was introduced by Gessel \cite{Ges}
and has since found many mathematical applications.  A link between poset Hopf algebras
and quasi-symmetric functions was found by Ehrenborg \cite{Ehrenborg} who encoded the flag vector of a graded
poset as a morphism from a Hopf algebra of graded posets to $\Qsym$.
Connections between structure coefficients and posets were extended further with the introduction
of Pieri operators on posets \cite{BMSvW}.  One central problem of algebraic combinatorics is to
show that certain symmetric functions are Schur positive.
Quasi-symmetric
functions have found an application as an intermediate step in a path to showing
the Schur positivity of symmetric functions (see for example \cite{HHLRU} and \cite{HHL}).

The graded dual Hopf algebra of the quasi-symmetric functions is the algebra of non-commutative 
symmetric functions $\Nsym$.  This Hopf algebra projects under the forgetful map onto $\sym$, and $\sym$ 
injects into the algebra $\Qsym$ of quasi-symmetric functions.
It was shown in \cite{ABS} that $\Nsym$ and $\Qsym$ 
are universal in the category of combinatorial Hopf algebras. They also play an important role in the 
representation theory of the 0-Hecke algebra (see \cite{KT} for details).  The algebra of
non-commutative symmetric functions $\Nsym$ also has applications to Kronecker products
and is closely related to the descent algebra \cite{MR}.  Recent connections have been made between
$\Nsym$ and mould calculus \cite{CHNT}.  The ubiquity of $\Qsym$ shows its importance in mathematics
and why we are interested in finding new bases with good structure with respect to the projection
of $\Nsym$ and the inclusion of $\sym$ into $\Qsym$.

Our program is to study new bases of $\Nsym$ and $\Qsym$ called {\it the immaculate basis}
and the {\it dual immaculate basis} (respectively) which were recently introduced in \cite{BBSSZ} as
an analogue of the Schur functions.  While Gessel's fundamental basis of $\Qsym$ (or the ribbon basis of $\Nsym$)
and the quasi-Schur basis \cite{HLMvW11a} seem to capture some of the Hopf algebra and combinatorial
structure that the Schur functions exhibit, no analogue seems to capture all properties and
it was never clear if any particular basis was completely canonical as an analogue of the Schur basis.
One motivation for further studying the immaculate basis is as a tool for resolving positivity conjectures in the 
space of symmetric functions.  In particular, the immaculate basis lifts the structure of the Schur
functions as a Jacobi-Trudi formula to the non-commutative and quasi-symmetric functions.

It was proved in \cite{BBSSZ} that the immaculate basis has many of 
the same properties of the Schur functions: 
it has a (positive) multiplicity free right-Pieri rule \cite[Theorem 3.5]{BBSSZ}, a simple 
Jacobi-Trudi determinant formula \cite[Theorem 3.23]{BBSSZ}, and they can be built by using an appropriate 
non-commutative Bernstein operator \cite[Definitions 3.1 and 3.2]{BBSSZ}.
By duality, the immaculate basis gives 
rise to a basis of $\Qsym$, which expands positively into the monomial and fundamental bases of $\Qsym$ 
\cite[Propositions 3.32 and 3.33]{BBSSZ}.  Moreover, the Schur function expansion of
a symmetric function can be read off of the expansion in the dual-immaculate basis and positivity of
a symmetric function expression in the dual immaculate basis implies positivity in the Schur basis 
\cite[Corollary 3.40]{BBSSZ}.

This paper is devoted to study of the {\it dual immaculate Pieri conjecture}
posted in \cite[Section 3.7.1]{BBSSZ} and to the {\it  left Pieri rule of immaculate}
as conjectured in \cite[Section 3.2]{BBSSZ}.  We  prove both of these
conjectures in this paper and provide a combinatorial interpretation for the coefficients which appear
in the product.

In particular we prove in Theorem \ref{compoThm} (where the result is stated in full detail),
\begin{Theorem} For a positive integer $s$ and a composition $\alpha$, 
\begin{equation} \label{formula}
F^\perp_s \fS_\alpha = \sum_{\gamma\models |\alpha|-s} (-1)^{\nn(\alpha-\gamma)} \fS_{\gamma}\,,
\end{equation}
where where $\nn(\alpha-\gamma)$ is the number of indices $i$ such that $\alpha_i<\gamma_i$
and the sum is over a subset of compositions $\gamma$ that satisfy $\ell(\gamma) = \ell(\alpha)$ or $\ell(\alpha)-1$.
\end{Theorem}

While the proof of this theorem is a bit technical, the idea is straightforward.  We find a combinatorial interpretation
for the terms which appear in an algebraic expression for $F_s^\perp \fS_\alpha$ that potentially have multiplicity,
then show precisely how the terms with duplicate indices cancel.

In Section~\ref{sec:prelim} we gather together some basic definitions and auxiliary results.
The equivalence of the Pieri rule for dual immaculate and left Pieri rules is shown in 
Section~\ref{sec:PieriD}. 
The precise statement and proof of Pieri rule for dual immaculate are found in Section~\ref{sec:proofD}.

\subsection{Acknowledgements}
This work is supported in part by  NSERC. The second author was supported by the Spanish MEC and Fondos FEDER 
jointly through project MTM2010-15223, and by the Junta de Andaluc\'ia (projects FQM-336 and FQM7156).
It is partially the result of a working session at the Algebraic
Combinatorics Seminar at the Fields Institute with the active
participation of C. Benedetti, O. Yacobi, E. Ens,  H. Heglin, D. Mazur and T. MacHenry.
We are also very thankful to Darij Grinberg for pointing out the 
equivalence of the the left Pieri rule and the dual Pieri rule, 
and for noticing errors in technical results in our earlier version.

This research was facilitated by computer exploration using the open-source
mathematical software \texttt{Sage}~\cite{sage} and its algebraic
combinatorics features developed by the \texttt{Sage-Combinat}
community~\cite{sage-co}.


\section{Preliminaries}\label{sec:prelim}

\subsection{Compositions, partitions and combinatorics}

In this subsection, we 
introduce some notation and definitions for partitions and compositions. 

A \textit{composition} $\alpha$ of a non-negative integer $n$ is a tuple 
$\alpha = [\alpha_1, \alpha_2, \ldots, \alpha_m]$ of positive integers 
such that $\sum^m_{i = 1} \alpha_i = n$. We  write 
$\alpha \models n$ to denote that $\alpha$ is a composition of $n$. The entries $\alpha_i$ of $\alpha$ are 
called the {\it parts} of $\alpha$. The {\it size} of $\alpha$ is the sum of the parts and is denoted 
$|\alpha|:=n$. The {\it length} of $\alpha$ is the number of parts and is denoted $\ell(\alpha):=m$. 
If the entries of $\alpha$ are weakly decreasing, i.e., $\alpha_1 \geq \alpha_2 \geq \ldots \geq \alpha_m$ 
then $\alpha$ is referred to as a {\it partition} of $n$, and we  write $\alpha \vdash n$. 

Let $n$ be a non-negative integer, we  represent a composition $\alpha = [\alpha_1, \alpha_2, \ldots, \alpha_m]$ 
of $n$ as a diagram of left justified rows of boxes or cells. We  follow the notation which states that the 
top row of the diagram has $\alpha_1$ cells, while the bottom row has $\alpha_m$ cells. 
For example, the composition
$[3,5,4,7,1,2]$ is represented as
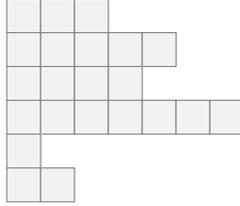
\[ 
\tikztableausmall{
{X, X, X},
{X, X, X, X, X}, 
{X, X, X, X}, 
{X, X, X, X, X, X, X}, 
{X}, 
{X, X}}
\]
For a composition $\alpha = [\alpha_1, \alpha_2, \ldots, \alpha_m]$ and a positive integer $s$, 
we write $[s, \alpha]$ to denote the composition $[s, \alpha_1, \alpha_2, \ldots, \alpha_m]$.
This convention also applies when we consider tuples of integers.

Compositions of $n$ are in bijection with subsets of $\{1, 2, \dots, n-1\}$. We  identify a 
composition $\alpha = [\alpha_1, \alpha_2, \dots, \alpha_m]$ with the subset $D(\alpha) = 
\{\alpha_1, \alpha_1+\alpha_2, \alpha_1+\alpha_2 + \alpha_3, \dots, \alpha_1+\alpha_2+\dots + \alpha_{m-1} \}$. 

Let $\alpha$ and $\beta$ be two compositions of $n$, we say that $\alpha \leq  \beta$ in {\it refinement order} 
if $D(\beta) \subseteq D(\alpha)$. For instance, $[1,1,2,1,3,2,1,4,2] \leq [4,4,2,7]$, since 
$D([1,1,2,1,3,2,1,4,2]) = \{1,2,4,5,8,10,11,15\}$ and $D([4,4,2,7]) = \{4,8,10\}$.


\subsection{Schur functions and creation operators}

The {\it algebra of symmetric functions} $\sym$ is the free algebra over $\QQ$ on commutative generators 
$\{h_1, h_2, \ldots \}$. The $h_n$ are usually called the complete homogeneous generators.
$\sym$ has a natural grading given by setting $h_i$ of degree $i$, and extending multiplicatively.  

Given an integer $n\ge 0$ and a partition $\lambda = [\lambda_1, \lambda_2, \ldots, \lambda_m]$ of $n$, 
the {\it complete homogeneous symmetric function} is the element of $\sym$ given by 
$h_\lambda:= h_{\lambda_1} \ldots h_{\lambda_m}$ with the convention that for $n=0$ 
we have a unique (empty) partition $\lambda=[]\vdash n$ and
$h_{[]}=1$. 
The set $\bigoplus_{n\geq0}\{h_\lambda\}_{\lambda \vdash n}$ constitutes a basis of $\sym$. Other  
common bases of $\sym$, that we  consider, are the following:
$e_\lambda$ the elementary;
$m_\lambda$ the  monomial;
$s_\lambda$ the Schur.
For simplicity, we let $h_i$, $e_i$, $m_i$, $p_i$ and $s_i$ denote the
corresponding generators indexed by the partition $[i]$.

It is known that $\sym$ is a self dual Hopf algebra, which has a pairing (the so-called Hall scalar product)
given by 
\[
\langle h_\lambda, m_\mu \rangle =
\langle s_\lambda, s_\mu \rangle = \delta_{\lambda, \mu}.
\]
An element $f \in \sym$ gives rise to an operator $f^\perp: \sym \to \sym$ 
according to the relation:
\begin{equation} \label{perpdef}
\langle fg, h \rangle = \langle g, f^\perp h \rangle \hspace{.1in} \textrm{ for all } g, h \in \sym.
\end{equation} 
Using \eqref{perpdef} as a definition, we can compute the action of the operator $f^\perp$ on an arbitrary element 
$g \in \sym$ by applying the following formula
\[
f^\perp(g) = \sum_{\lambda} \langle g, f a_\lambda \rangle b_\lambda,
\]
where $\{ a_\lambda \}_{\lambda}$ and $\{ b_\lambda \}_\lambda$ are any two
bases of $\sym$ which are dual with respect to the pairing $\langle\cdot,\cdot\rangle$.

We define a ``creation'' operator $\B_m: \sym_n \to \sym_{m+n}$ by: 
\[
\B_m := \sum_{i \geq 0} (-1)^i h_{m+i} e_i^\perp.
\] 

Inspired by the following theorem, which states that creation operators recursively define Schur functions, 
the immaculate basis of $\Nsym$ was introduced in \cite{BBSSZ} (see Definition \ref{def:immaculate}).

\begin{Theorem} {\rm (Bernstein \cite[pg 69-70]{Zel})} \label{th:bern}
For all tuples $\alpha \in \ZZ^m$, 
\[ s_\alpha :=  \det| h_{\alpha_i + j - i } |_{1 \leq i,j \leq \ell(\alpha)} 
= \B_{\alpha_1} \B_{\alpha_2} \cdots \B_{\alpha_{m}}(1).\]
\end{Theorem}


\subsection{Non-commutative symmetric functions}

The {\it algebra of non-commutative symmetric functions} $\Nsym$, a non-commutative analogue of $\sym$, arises by
considering an algebra with one non-commutative generator at each positive
degree. More precisely, we define $\Nsym$ as the algebra with generators $\{\HH_1, \HH_2, \dots \}$ and 
no relations. We endow $\Nsym$ with the structure of a graded algebra by setting each generator  
$H_i$ of degree $i$. We write $\Nsym_n$ to denote the 
graded component of $\Nsym$ of degree $n$. 

Given a composition $\alpha = [\alpha_1, \alpha_2, \dots, \alpha_m]$, the \textit{complete homogeneous function} 
associated to $\alpha$ is defined as $\HH_\alpha := \HH_{\alpha_1} \HH_{\alpha_2} \ldots \HH_{\alpha_m}$. 
Let $n$ be a nonnegative integer, the set $\{\HH_\alpha \}_{\alpha \vDash n}$ indexed by compositions of $n$ 
is a basis of $\Nsym_n$. To make this notation consistent, some formulas  use expressions that have 
$H$ indexed by tuples of integers
and we use the convention that $\HH_0 = 1$ and $\HH_{-r} = 0$ for every $r > 0$.

The {\it forgetful map} $\chi: \Nsym \to \sym$ is defined by sending the basis element  
$\HH_\alpha$ to the complete homogeneous symmetric function 
\begin{equation}\label{def:chi}
\chi(\HH_\alpha) := h_{\alpha_1} h_{\alpha_2} \cdots h_{\alpha_{\ell(\alpha)}} \in \sym,
\end{equation}
and extend it linearly to any element of $\Nsym$. Note that $\chi$ is a surjection onto $\sym$. 
The image by $\chi$ of a non-commutative symmetric function  is called its commutative image.



\subsection{Quasi-symmetric functions}

The algebra of quasi-symmetric functions $\Qsym$ was introduced in \cite{Ges} 
(see also subsequent references such as \cite{GR, Sta84}). We  
realize $\Qsym$ as the graded Hopf algebra dual to $\Nsym$ which contains $\sym$ as a subalgebra.

If $\Nsym_n$ is the homogeneous subspace of $\Nsym$ spanned by $\{ H_\alpha : \alpha \models n\}$,
then $\Qsym_n$ is the space dual to $\Nsym_n$ and
$$\Qsym = \bigoplus_{n\geq0} \Qsym_n\,.$$
We let $\{M_\alpha: \alpha \models n\}$ 
denote the monomial quasisymmetric basis, dual to the $\{ H_\alpha : \alpha \models n\}$ basis.


Note that $\sym$ can be seen as a subalgebra of $\Qsym$.  In fact, given a partition $\lambda$, 
the monomial symmetric function $m_\lambda \in \sym$ can 
be written in terms of the quasi-symmetric monomial functions as:
\[ 
m_\lambda = \sum_{\sort(\alpha) = \lambda} M_\alpha,
\]
where $\sort(\alpha)$ is the partition obtained by reordering the parts of 
$\alpha$ from the largest to the smallest.

The {\it fundamental quasi-symmetric function} $F_\alpha$, indexed by a composition $\alpha$, 
is defined by the following expansion in the monomial quasi-symmetric basis:
\[
F_\alpha = \sum_{\beta \leq \alpha} M_\beta.
\]
The set $\{F_\alpha\}_{\alpha \models n}$ forms a basis of $\Qsym_n$.  In particular we have
$$F_r = \sum_{\alpha \models r} M_\alpha = h_r
\hbox{ and }
F_{1^r} = M_{1^r} = e_r~.$$


\subsection{Identities relating $\Nsym$ and $\Qsym$}
The algebras $\Nsym$ and $\Qsym$ form graded dual Hopf algebras. The monomial basis of $\Qsym$ 
is dual to the complete homogeneous basis of $\Nsym$.
$\Nsym$ and $\Qsym$ have a pairing $\langle \cdot, \cdot \rangle: \Nsym \times \Qsym \to \mathbb{Q}$, 
defined under this duality where $\langle \HH_\alpha, M_\beta \rangle = \delta_{\alpha, \beta}$.

The operation adjoint to multiplication with respect to the scalar product on $\sym$
can be generalized to $\Nsym$ and $\Qsym$ by using the pairing above. Note that this new operation 
is dual to the multiplication by a quasi-symmetric function. More precisely, for $F \in \Qsym$, 
we denote $F^\perp$ the operator which acts on elements $H \in \Nsym$ according to the following relation 
\begin{equation} \label{perpNsym}
\langle H, F G \rangle = \langle F^\perp H, G \rangle \quad \mbox{for all} \quad G \in \Qsym.
\end{equation}
Given $H \in \Nsym$, as was done with $\sym$, to compute $F^\perp(H)$ we take a basis $\{ A_\alpha \}_\alpha$ 
of $\Qsym$ and $\{ B_\alpha \}_\alpha$ a basis
of $\Nsym$ such that $\langle B_\alpha, A_\beta \rangle = \delta_{\alpha\beta}$, then by \eqref{perpNsym} we get that
\[
F^\perp(H) = \sum_{\alpha} \langle H, F A_\alpha \rangle B_\alpha.
\]



\subsection{The immaculate basis of $\Nsym$}
The immaculate basis of $\Nsym$ was introduced in \cite[Section 3]{BBSSZ}. The authors obtained these functions 
as a non-commutative analogue of the Jacobi-Trudi identity (see \cite[Theorem 3.23]{BBSSZ}). 
The immaculate functions have also been defined as the unique functions in $\Nsym$ which satisfy a given positive 
multiplicity free right-Pieri rule (see \cite[Theorem 3.5]{BBSSZ}). In this paper, 
it is more useful for our purposes to choose the construction of the immaculate basis based on a 
non-commutative version of the Bernstein operators. 

\begin{Definition} \cite{BBSSZ}
For $m \in \ZZ$, the {\it non-commutative Bernstein operator} $\BB_m$ is defined as:
\[ 
\BB_m := \sum_{i \ge 0} (-1)^i \HH_{m+i}^L F_{1^i}^\perp.
\]
where we consider $H_m^L$ as the operator on $\Nsym$ that left multiplies by $H_m$.
\end{Definition}

\begin{Remark}
Note that $\BB_m$ clearly deserves to be the non-commutative sibling of the Bernstein operator $\B_m$, 
since the symmetric function $e_i$ is equal to $F_{1^i}$ under the identification of $\sym$ as a 
subalgebra of $\Qsym$.
\end{Remark} 

As expected, in the same way that creation operators are used to construct Schur functions 
(see Theorem \ref{th:bern}), these non-commutative Bernstein operators can be used to 
inductively build non-commutative functions.

\begin{Definition} \cite{BBSSZ} \label{def:immaculate}
For any $\alpha = (\alpha_1, \alpha_2, \cdots, \alpha_m) \in \ZZ^m$, 
the \emph{immaculate function} $\fS_\alpha \in \Nsym$ 
is defined as the composition of the non-commutative Bernstein operators indexed by the entries in 
$\alpha$ applied to $1$. More precisely:
\[
\fS_\alpha := \BB_{\alpha_1} \BB_{\alpha_2} \cdots \BB_{\alpha_m} (1).
\]
\end{Definition}
It was shown in \cite{BBSSZ} that the elements
$\{ \fS_\alpha \}_{\alpha \models n}$ form a basis for $\Nsym_n$, which project onto Schur functions, that is, 
$\chi(\fS_\alpha) = s_\alpha$.

\subsection{Identities relating the operators $\BB_m$ and $F^{\perp}_r$} \label{identities}  
Some of the identities that we use in this paper appear in \cite{BBSSZ}.  In particular, we  use the
following results which can be derived from equations in that paper.

\begin{Proposition}\label{genidents}  For an integer $m \in {\mathbb Z}$ and $r \geq 0$,
\begin{equation} \label{eq:leftskew}
F_r^\perp \BB_m = \sum_{j=0}^r \BB_{m-j} F_{r-j}^\perp
\end{equation}
and
\begin{equation}
H_m^L = \sum_{r \geq 0} \BB_{m+r} F_r^\perp~. \label{eq:leftmult}
\end{equation}
\end{Proposition}
\begin{proof}
Using \cite[Lemma 2.4]{BBSSZ} we obtain that $F_r^\perp H_m^L = \sum_{j\ge0}^r H_{m-j}^L F_{r-j}^\perp$.
The relation~\eqref{eq:leftskew} is obtained as follows.
\begin{align*}
F_r^\perp \BB_m 
&= \sum_{i \geq 0} F_r^\perp (-1)^i H_{m+i}^L F_{1^i}^\perp 
= \sum_{i \geq 0} \sum_{j \geq 0}^r (-1)^i H_{m+i-j}^L F_{r-j}^\perp F_{1^i}^\perp\\
&= \sum_{j \geq 0}^r \sum_{i \geq 0} (-1)^i H_{m+i-j}^L F_{1^i}^\perp F_{r-j}^\perp 
= \sum_{j \geq 0}^r \BB_{m-j} F_{r-j}^\perp\,.
\end{align*}
Equation~\eqref{eq:leftmult} follows since 
$\sum_{i+j=d} (-1)^j F_i F_{1^j} = \sum_{i+j=d} (-1)^j h_i e_{j} = \delta_{d,0}$ and
\begin{align*}
\sum_{r \geq 0} \BB_{m+r} F_r^\perp = \sum_{r \geq 0} \sum_{i \geq 0} H_{m+r+i}^L (-1)^i F_{1^i}^\perp F_r^\perp 
= \sum_{d \geq 0} \sum_{a+b = d} H_{m+d}^L (-1)^a F_{1^a}^\perp F_b^\perp = H_m^L\,.
\end{align*}
\end{proof}

We start by analyzing the relationship between the non-commutative Bernstein 
operators and the dual (skew) operators.  
As we will see the resulting identities will play a crucial role in the proof of the 
Pieri rule for the dual immaculate basis.
The following identities follow as special cases of the identities in Proposition \ref{genidents}.

\begin{Corollary} The following relations hold for any integers $s>0$, $m>0$ and $r\ge 0$
\begin{equation}\label{eq:rel1}
\sum_{i\geq 0} \BB_i F^\perp_i = \Id~,
\end{equation}
\begin{equation}\label{eq:rel2}
\sum_{i\geq -s} \BB_i F^\perp_{i + s} = 0~,
\end{equation}
\begin{equation}\label{eq:rel3}
\BB_0 = \Id - \sum_{i\geq 1} \BB_i F^\perp_i~,
\end{equation}
\begin{equation}\label{eq:rel4}
\BB_{-s} = - \sum_{i > -s} \BB_i F^\perp_{i + s}~,
\end{equation}
\begin{equation}\label{eq:rel5}
F^\perp_r \BB_m = \sum^r_{i = 0} \BB_{m - i} F^\perp_{r - i} \rlap{\qquad   if  $m>r$~,}
\end{equation}
\begin{equation}\label{eq:rel6}
F^\perp_m \BB_m = \Id -\sum_{i \geq 1} \BB_{m + i} F^\perp_{m + i}~,
\end{equation}
\begin{equation}\label{eq:rel7}
F^\perp_r \BB_m = -\sum_{i \geq 1} \BB_{m + i} F^\perp_{r + i} \rlap{\qquad   if  $m<r$~.}
\end{equation}
\end{Corollary}

\begin{proof} Equation~\eqref{eq:rel1} to~\eqref{eq:rel5} are special cases of Proposition~\ref{genidents}. 
Equation~\eqref{eq:rel6} follows using~\eqref{eq:leftskew} and~\eqref{eq:rel3},
 $$F^\perp_m \BB_m  = \sum_{j\geq0}^m \BB_{m-j} F_{m-j}^\perp = 
 \BB_0+\sum_{j\geq0}^{m-1} \BB_{m-j} F_{m-j}^\perp =  \Id -\sum_{i \geq 1} \BB_{m + i} F^\perp_{m + i}.$$
Similarly, for Equation~\eqref{eq:rel7} we have $m-r<0$. Using~\eqref{eq:leftskew} and~\eqref{eq:rel4}
 $$ F^\perp_r \BB_m = \sum_{j\geq0}^r \BB_{m-j} F_{r-j}^\perp 
 = \BB_{m-r}+ \sum_{j\geq0}^{r-1} \BB_{m-j} F_{r-j}^\perp
 = - \sum_{i > -r} \BB_{m+i} F^\perp_{i + r} +  \sum_{i>-r}^{0} \BB_{m+i} F_{r+i}^\perp$$
\end{proof}

\begin{Remark}
Given $r\geq 0$, Pieri rules for the skew operators $F^\perp_{1^r}$, $F^\perp_r$ were developed in 
\cite[3.6]{BBSSZ}. For the second operator, the following formula was shown in \cite[Proposition 3.35]{BBSSZ}:
\begin{equation} \label{SkewPieri}
F_r^\perp \fS_\alpha = \sum_{\substack{\beta \in \ZZ^{m}\\
i-m\leq\beta_i\leq\alpha_i\\|\beta|=|\alpha|-r}}
\fS_\beta,
\end{equation}
where $\alpha \in \ZZ^m$. 
Note that the immaculate functions appearing on the right hand side of 
\eqref{SkewPieri} need not to be basis elements, 
even in the case that $\alpha$ is a partition. For example, for $r = 2$ and $\alpha = [2, 1]$ 
the formula \eqref{SkewPieri} applies to get that 
\[
F_2^\perp \fS_{21} =  \fS_{01} +  \fS_{10},
\]
and  $\fS_{01}$ is not a basis element. 
Here, we are interested in expressing $F^\perp_r \fS_\alpha$ in terms of the immaculate basis
in a way that shows that the coefficients are $\pm 1$. 
In the example above, our formulae will apply to get that $F_2^\perp \fS_{21} =  \fS_1$. 

Formula~\eqref{SkewPieri} is deduced using relation~\eqref{eq:leftskew} only
and the expression can potentially involve operators $\BB_m$ with $m\le 0$ .
In Section~\ref{sec:proofD} we will make use of the relations~\eqref{eq:rel5}, 
\eqref{eq:rel6} and \eqref{eq:rel7} which involve only operators
$\BB_m$ with $m>0$. This will allow us to compute the exact coefficient in the immaculate basis.
\end{Remark} 



\subsection{The dual immaculate basis}  
Every basis $\{X_\alpha\}$ of $\Nsym_n$ gives rise to a basis $\{Y_\beta\}$ of $\Qsym_n$ defined by duality.
The elements $\{Y_\beta\}$ are the unique elements of $\Qsym_n$ satisfying 
$\langle X_\alpha, Y_\beta \rangle = \delta_{\alpha,\beta}$. 
The dual basis to the immaculate basis of $\Nsym$, 
denoted $\fS_\alpha^*$, was studied in \cite[Section 3.7]{BBSSZ}. 
More properties of dual immaculate functions have been investigated in \cite{BBSSZ1,BBSSZ2}. 


\section{The Pieri conjectures}\label{sec:PieriD}

In \cite[Theorem 3.5]{BBSSZ} it was shown that the immaculate basis satisfies 
a positive multiplicity free right Pieri rule. 
Concerning to the left Pieri rule, note that the products of the form 
$H_m \fS_\alpha$ can have negative signs in their expansion in terms of the immaculate basis.  They also conjectured
\cite[Conjecture 3.7]{BBSSZ} that there is a statistic $sign(\alpha, \beta)$ such that
\begin{equation} \label{leftPieriConj}
H_s \fS_\alpha = \sum_\beta (-1)^{sign(\alpha, \beta)}\fS_\beta
\end{equation}
for some collection of compositions $\beta$.
For instance,
\[ 
\HH_2 \fS_{14} = \fS_{214} - \fS_{322} - \fS_{331} - \fS_{421} - \fS_{43} - \fS_{52}. 
\]
It was conjectured in \cite[Remark 3.6]{BBSSZ} that the left Pieri rule is multiplicity free, up to sign.

Similarly, a conjecture about the Pieri rule for the 
dual immaculate basis was stated in \cite[Section 3.7.1]{BBSSZ}. To be more precise: 
\begin{equation} \label{pieridIConj}
F_s \fS_\alpha^* = \sum_\beta (-1)^{sign'(\alpha, \beta)}\fS_\beta^*
\end{equation}
for some collection of $\beta$ and some statistic $sign'$. 
Recall that on the quasisymmetric side $F_s=h_s$. For example, for $s = 2$ and $\alpha = [2,1,2]$ we have that
\begin{equation}\label{eq:examp1}
F_2 \fS^*_{212} = 
- \fS^*_{1312} 
- \fS^*_{142} 
+ \fS^*_{2212}
+ \fS^*_{3112}
+ \fS^*_{322}
+ \fS^*_{412}.
\end{equation}
Our main goal is to prove \eqref{pieridIConj}. The equivalence of \eqref{leftPieriConj} 
and \eqref{pieridIConj} is given by the following computation.
In a private communtication, Darij Grinberg gave us a first proof of the equivalence
and here we provide a different argument which follows from our operator relations.

\begin{Lemma}
For $s \geq 0$ and $r>0$ and for compositions $\alpha$ and $\beta$ such that
$|\alpha| = |\beta| + s$, we have
\begin{equation}\label{eq:transfer}
\left< \fS_\alpha, F_s \fS_\beta^\ast \right> = \left< H_{r} \fS_\alpha, \fS_{(s+r,\beta)}^\ast \right>~.
\end{equation}
\end{Lemma}
\begin{proof}
First apply \eqref{eq:leftmult} to show
\begin{align}\label{eq:firsthalf}
\left< H_r \fS_\alpha, \fS^\ast_{(r+s, \beta)} \right>
= \left< \sum_{j \geq 0} \BB_{r+j} F_j^\perp \fS_\alpha, \fS^\ast_{(r+s,\beta)} \right>~.
\end{align}
Now the only terms with the first part of the indexing composition equal to $r+s$ in the expansion of
$\BB_{r+j} F_j^\perp \fS_\alpha$ are those with $j=s$.  In this case, the coefficient
of $\fS_{(r+s,\beta)}$ in $\BB_{r+s} F_s^\perp \fS_\alpha$ is equal to the coefficient
of $\fS_\beta$ in $F_s^\perp \fS_\alpha$.  Hence \eqref{eq:firsthalf} implies
\begin{align*}
\left< H_r \fS_\alpha, \fS^\ast_{(r+s, \beta)} \right>
= \left< \BB_{r+s} F_s^\perp \fS_\alpha, \fS^\ast_{(r+s,\beta)} \right>
= \left<  F_s^\perp \fS_\alpha, \fS^\ast_{\beta} \right>
= \left<  \fS_\alpha, F_s  \fS^\ast_{\beta} \right>~.
\end{align*}
\end{proof}


Equation \eqref{eq:leftmult} also states that all terms in the expansion of $H_r \fS_\alpha$ have
the first part of the indexing composition greater than or equal to $r$.
Therefore, Equation \eqref{eq:transfer} shows the coefficient in the expansion
$H_r \fS_\alpha$ are the same as the coefficients in the expansion of $F_s^\perp \fS_\beta$. 
Thus, from this point on, we concentrate on proving \eqref{pieridIConj} by giving
an expansion of $F_s^\perp \fS_\alpha$ since by duality, the coefficient of $\fS_\beta^*$ 
in the expansion of $F_s \fS_\alpha^*$ coincides with the coefficient of $\fS_\alpha$ in the expansion of 
$F^\perp_s \fS_\beta$.

In the example in Equation \eqref{eq:examp1}, we can assert that the coefficients of 
$\fS_{212}$ in the expansions of 
$F^\perp_2 \fS_{1312}$, $\, \, F^\perp_2 \fS_{142}$, $\, \, F^\perp_2 \fS_{2212}$, $\, \, F^\perp_2 \fS_{3112}$, 
$\, \, F^\perp_2 \fS_{322}$ and $\, \, F^\perp_2 \fS_{412}$ are $-1, \, -1, \, 1, \, 1, \, 1, \, 1$, 
respectively. In fact:
\[
F^\perp_2 \fS_{1312}  = - \fS_{212}, \qquad F^\perp_2 \fS_{142} = - \fS_{212} - \fS_{221} - \fS_{23} - \fS_{32}, 
\qquad F^\perp_2 \fS_{2212}  = \fS_{212},
\]
\[
F^\perp_2 \fS_{3112} = \fS_{1112} +	\fS_{212}, \qquad F^\perp_2 \fS_{322} 
= \fS_{122} +	\fS_{212} + \fS_{221} +	\fS_{32}, 
\qquad F^\perp_2 \fS_{412} = \fS_{212},
\]
and the immaculate function $\fS_{212}$ does not appear in the expansion of $F^\perp_2 \fS_{\alpha}$, where $\alpha$ a composition of 7 distinct from the previous ones. 

Thus, in order to prove \eqref{pieridIConj} it is enough to show that for any composition $\alpha$ 
and any positive integer $s$, the noncommutative function $F^\perp_s \fS_\alpha$ has 
a signed multiplicity free expansion in the immaculate basis. Moreover, we  also prove that the expansion 
of $F^\perp_s \fS_\alpha$ is positive when $\alpha$ is a partition.





\section{Proof of the Pieri rule for dual immaculate}\label{sec:proofD} 

As seen in the last example, in general, the application of a skew operator $F^\perp_s$ to an immaculate function 
does not expand positively in the immaculate basis. In the present section, 
we  prove that $F^\perp_s \fS_\lambda$ expands positively in the immaculate basis provided that 
$\lambda$ is a partition. For a composition $\alpha$, we  show that the expansion of 
$F^\perp_s \fS_\alpha$ is multiplicity free.  In other words, the nonzero coefficients of 
$F^\perp_s \fS_\alpha$ in the immaculate basis are either 1 or $-1$.

This is achieved by using the relations~\eqref{eq:rel5}, \eqref{eq:rel6} and \eqref{eq:rel7}. 
Given a composition $\alpha=[\alpha_1,\alpha_2, \ldots, \alpha_m]$
we  write 
$\tail(\alpha)$ to denote the composition obtained from $\alpha$ by dropping the first entry, namely, 
$\tail(\alpha) = [\alpha_2, \ldots, \alpha_m]$.  For $m=1$, $tail(\alpha) = []$. We have
  $$F^\perp_s \fS_\alpha = F^\perp_s \BB_\alpha (1) = F^\perp_s \BB_{\alpha_1} \BB_{\tail(\alpha)} (1) \,.$$
Comparing $s$ and $\alpha_1$ the relations~\eqref{eq:rel5}, \eqref{eq:rel6} and \eqref{eq:rel7} give us three cases:
\begin{equation}\label{eq:Fperples}
F^\perp_s \BB_{\alpha_1} \BB_{\tail(\alpha)} (1)  
= \sum^s_{\beta_1 = 0} \BB_{\alpha_1 - \beta_1} F^\perp_{s-\beta_1}\BB_{\tail(\alpha)} (1) 
\rlap{\qquad if $\alpha_1>s$,}
\end{equation}
\begin{equation}\label{eq:Fperpeq}
F^\perp_s \BB_{\alpha_1} \BB_{\tail(\alpha)} (1) 
= F^\perp_0\BB_{\tail(\alpha)} (1) 
-\sum_{\beta_1 < 0} \BB_{\alpha_1 -\beta_1} F^\perp_{s - \beta_1} \BB_{\tail(\alpha)} (1)
\rlap{\qquad if $\alpha_1=s$,}
\end{equation}
\begin{equation}\label{eq:Fperpmore}
F^\perp_s \BB_{\alpha_1} \BB_{\tail(\alpha)} (1)  
= -\sum_{\beta_1 <0} \BB_{\alpha_1 - \beta_1} F^\perp_{s - \beta_1}\BB_{\tail(\alpha)} (1) 
\rlap{\qquad if $\alpha_1<s$.}
\end{equation}
In Equation~\eqref{eq:Fperpeq} we used $F^\perp_0$ to represent $\Id=F^\perp_0$. 
If we continue to commute the operators $F^\perp$ and $\BB$
we get a formula of the form
\begin{equation}\label{eq:Fperpdev}
F^\perp_s \fS_\alpha = F^\perp_s \BB_{\alpha} (1)  
= \sum_{\beta}\sgn(\beta) \BB_{\comp(\alpha - \beta)} F^\perp_{s - |\beta|}(1) \,,
\end{equation}
where the sum is over $\beta \in \ZZ^m$ that record the change in the index of the $F^\perp$ operator
as it commutes past $\BB_{\alpha_i}$. We represent $\sgn(\beta) = (-1)^{\nn(\beta)}$ with $\nn(\beta)$ 
the number of negative entries in $\beta$, and $\comp(\alpha - \beta)$ is the composition obtained 
from $\alpha - \beta=[\alpha_1-\beta_1,\alpha_2-\beta_2,\ldots,\alpha_m-\beta_m]$ by removing the 
parts that are $0$. We now remark that $F^\perp_{s - |\beta|}(1)=0$ unless $s-|\beta|=0$. 
Furthermore, in the $i$th step of commutation between 
$F^\perp_{s - \beta_1-\beta_2-\cdots-\beta_{i-1}}$ and $\BB_{\alpha_i}$,
we must have $s - \beta_1-\beta_2-\cdots-\beta_{i-1}\ge 0$ since $F_d = 0$ if $d<0$.
 
With the above notation in mind, we define the set $\Z_{s,\alpha}$ consisting of all sequences of integers 
$\beta = (\beta_1, \beta_2, \ldots, \beta_m)$ satisfying the following conditions: 
\begin{itemize}\itemsep = 1mm
\item[\rm(Z1)] $ \beta_1 + \ldots + \beta_m = s$  and $ \beta_1 + \ldots + \beta_i \le s$ for $i<m$,
\item[\rm(Z2)] $\alpha_i - \beta_i\ge 0$ for $1\le i\le m$ and 
$\alpha_i - \beta_i= 0$ for {\bf at most} one $i$ in  $1\le i\le m$
\item[\rm(Z3)] for $i \in \{1, \ldots, m\}$,
 \begin{itemize}\itemsep = 1mm
 \item[\rm(a)]   if $\, \alpha_i > s - (\beta_1 + \ldots + \beta_{i-1})$, then 
 $\, \, 0 \leq \beta_i \leq s - (\beta_1 + \ldots + \beta_{i-1}), \, \,$
 \item[\rm(b)]   if $\, \alpha_i = s - (\beta_1 + \ldots + \beta_{i-1})$, then
  \begin{itemize}\itemsep = 1mm
  \item[\rm(1)] $\, \beta_i = \alpha_i \, \, \mbox{ and } \, \, \beta_{i+1} = \ldots = \beta_m = 0, \, \,$ or 
  \item[\rm(2)] $\, \beta_i < 0, \, \,$ 
  \end{itemize}
  \item[\rm(c)] if $\, \alpha_i < s - (\beta_1 + \ldots + \beta_{i-1})$, then $\, \, \beta_i < 0$.
  \end{itemize}
\end{itemize}
The condition (Z1) is clear from the paragraph above. The conditions in (Z3) depend on which case we use in the 
$i$th step of commutation.
Condition~(Z3a) corresponds to Equation~\eqref{eq:Fperples} and condition (Z3c) corresponds to 
Equation~\eqref{eq:Fperpmore}. For condition~(Z3b)  we have to consider two subcases
depending if we have $\beta_i<0$ (the summation on the right hand side of Equation~\eqref{eq:Fperpeq}) 
or $\beta_i=\alpha_i$ in which case
$F^\perp_{s-\beta_1-\ldots-\beta_i}=F^\perp_0$ (the first term of the right hand side of 
Equation~\eqref{eq:Fperpeq}). We remark that  $\beta_i=\alpha_i$ can occur only once for a given 
$\beta$ and once it occurs, the part $\BB_{\alpha_i}$ is removed and 
$F^\perp_{s-\beta_1-\ldots-\beta_i}=F^\perp_0$ commutes with the remaining $\BB$. 
This explains  condition (Z2) and the fact that $\beta_{i+1} = \ldots = \beta_m = 0$ in this case.
To clarify, we have that every $\alpha_i \geq \beta_i$ 
and there is at most one index $i$ such that $\alpha_i = \beta_i$.
This justifies the definition
of $\comp(\alpha - \beta)$ in Equation~\eqref{eq:Fperpdev}.

Negative terms are introduced only in Equation~\eqref{eq:Fperpmore} and in the summation 
of Equation~\eqref{eq:Fperpeq}. The total sign contribution for a given $\beta$ is exactly the number 
of times (Z3c) or (Z3b2) is involved in each step. That is exactly  the number of negative entries in 
$\beta$, and we now justify why $ \sgn(\beta) = (-1)^{\nn(\beta)}$ in Equation~\eqref{eq:Fperpdev}. 
A direct consequence of Equation~\eqref{eq:Fperpdev} is the following proposition.

\begin{Proposition} \label{Prop:partialPieri} 
\begin{equation}\label{eq:FperpImm}
F^\perp_s \fS_\alpha = \sum_{\beta\in \Z_{s,\alpha}}\sgn(\beta) \fS_{\comp(\alpha - \beta)} \,.
\end{equation}
\end{Proposition}

\begin{Remark}
Proposition~\ref{Prop:partialPieri} is not a cancelation free formula and hence 
it does not show yet the desired result. 
For example, if one has $\alpha=[5,1,3,7]$
and $s=2$, then the vectors $(1,1,0,0)$, $(1,-2,3,0)$ and $(1,-2,-4,7)$ belong to $\Z_{s,\alpha}$. One has 
  $$[4,3,7]=\comp\big(\alpha-(1,1,0,0)\big)=\comp\big(\alpha-(1,-2,3,0)\big)= \comp\big(\alpha-(1,-2,-4,7)\big).$$
Notice then that $1=\sgn(1,1,0,0)=-\sgn(1,-2,3,0)=\sgn(1,-2,-4,7)$. Two of these terms cancel and the coefficient of 
$\fS_{[4,3,7]}$ in Equation~\eqref{eq:FperpImm} with $s=2$ and $\alpha=[5,1,3,7]$ is
$1$.
\end{Remark}

We now concentrate the coefficient of $\fS_\gamma$ in Equation~\eqref{eq:FperpImm} by considering the
following subset of $\Z_{s,\alpha}$. Let
  $$\Z_{s,\alpha}^\gamma=\big\{ \beta\in \Z_{s,\alpha} \big| \comp(\alpha-\beta)=\gamma\big\}\,.$$
Equation~\eqref{eq:FperpImm}  can now be written as
\begin{equation}\label{eq:FperpImm2}
F^\perp_s \fS_\alpha = \sum_{\gamma\models |\alpha|-s} 
\left( \sum_{\beta\in \Z_{s,\alpha}^\gamma}\sgn(\beta)\right) \fS_{\gamma} \,.
\end{equation}
Hence to find a cancellation free formula, we focus our attention on the set $\Z_{s,\alpha}^\gamma$. 
To begin, we remark that in view of condition~(Z2) we have
	$$ \ell(\alpha)-1 \le \ell(\comp(\alpha-\beta)) \le \ell(\alpha)\,.$$
If $\ell(\gamma)=\ell(\alpha)$, then clearly $\beta=\alpha-\gamma$ is uniquely determined and if such 
$\beta\in\Z_{s,\alpha}$ then the condition~(Z3b1) does not occur for such $\beta$. We thus have
	$$ \ell(\gamma)=\ell(\alpha) \quad\implies\quad \big|\Z_{s,\alpha}^\gamma\big|\le 1\,.$$

\begin{Proposition} \label{prop:coefset}
Given $s>0$, $\, \alpha=[\alpha_1,\ldots,\alpha_m]$ and $\gamma\models |\alpha|-s$, we have
\begin{itemize}\itemsep = 1mm
\item[\rm (A1)] If $\ell(\gamma)<\ell(\alpha)-1$ or $\ell(\gamma)>\ell(\alpha)$, then \quad 
$\Z_{s,\alpha}^\gamma=\emptyset$;
\item[\rm (A2)] If $\ell(\gamma)=\ell(\alpha)$ and $\Z_{s,\alpha}^\gamma\ne \emptyset$, then \quad 
$\Z_{s,\alpha}^\gamma=\big\{ \alpha-\gamma \big\}$;
\item[\rm (A3)] If $\ell(\gamma)=\ell(\alpha)-1$ and $\Z_{s,\alpha}^\gamma\ne \emptyset$, then let 
$1\le k\le m$ be the smallest integer such that $\alpha_i=\gamma_{i-1}$ for all $k<i\le m$ where 
$k=m$ when $\alpha_m\ne\gamma_{m-1}$. Let $k\le r\le m$ be the largest integer such that 
$\alpha_k <\alpha_{k+1}<\cdots<\alpha_r$. We have in this case
\begin{align*}\Z_{s,\alpha}^\gamma=\big\{ 
  	&(\alpha_1-\gamma_1,\ldots,\alpha_{k-1}-\gamma_{k-1},\alpha_k,0,\ldots,0), \\
	&(\alpha_1-\gamma_1,\ldots,\alpha_{k-1}-\gamma_{k-1},\alpha_k-\alpha_{k+1},\alpha_{k+1},0,\ldots,0),\\
	&\cdots \\
	&(\alpha_1-\gamma_1,\ldots,\alpha_{k-1}-\gamma_{k-1},\alpha_k-\alpha_{k+1},\ldots,
	\alpha_{r-1}-\alpha_{r},\alpha_r,0,\ldots,0)
  \big\}
  \end{align*}
\end{itemize}
\end{Proposition}	
	
\begin{proof}
The cases (A1) and (A2) are clear from the discussion before the statement of the Proposition. In the case~(A3), 
using conditions~(Z2) and (Z3b1), we know that if $\Z_{s,\alpha}^\gamma\ne \emptyset$,
then for each $\beta\in \Z_{s,\alpha}^\gamma$ exactly one entry  $\beta_j=\alpha_j$ for some 
$1\le j\le m$ and $\beta_i=0$ for all $j<i\le m$. That is
   \begin{equation}\label{eq:mybeta} 
        \beta=(\beta_1,\ldots,\beta_{j-1},\alpha_j,0,\ldots,0).
   \end{equation}
This means
    \begin{equation}\label{eq:compbeta}  
         \gamma =\comp(\alpha-\beta)=[\alpha_1-\beta_1,\ldots,\alpha_{j-1}-\beta_{j-1},\alpha_{j+1},
         \ldots,\alpha_m]\,.
    \end{equation}
What we will do is show that $j$ must be between $k$ and $r$ in order for this to happen.
    
Let $\beta\in \Z_{s,\alpha}^\gamma$ and $j$ as in Equation~\eqref{eq:mybeta}. If $\alpha_j<\alpha_{j+1}$, then
\begin{equation}\label{eq:betaprime}
\beta\in \Z_{s,\alpha}^\gamma \quad\implies\quad \beta'=(\beta_1,\ldots,\beta_{j-1},\alpha_j-\alpha_{j+1},
\alpha_{j+1},0,\ldots,0)\in \Z_{s,\alpha}^\gamma\,.
\end{equation}
This follows from the fact that  if the conditions (Z1)--(Z3) are satisfied for $\beta_i$ for $1\le i\le m$, 
then the conditions (Z1)--(Z3) are also trivially satisfied for $\beta'_i$ for $1\le i<j$ and for $j+1<i\le m$. 
For the step $j$, we have $\beta_j=\alpha_j$ and this implies we are in condition~(Z3b1). 
Hence $\alpha_j=s-\beta_1-\cdots-\beta_{j-1}$. Instead of using (Z3b1) we could use (Z3b2) and $\beta'_j$ 
could be any negative integer.
We choose $\beta'_j=\alpha_j-\alpha_{j+1}<0$ since $\alpha_j<\alpha_{j+1}$. 
Now in step $j+1$ for $\beta'_{j+1}$, we have 
  $$s-\beta'_1-\cdots-\beta'_j= s-\beta_1-\cdots-\beta_{j-1} -\alpha_j+\alpha_{j+1} 
  =\alpha_j-\alpha_j+\alpha_{j+1}=\alpha_{j+1}\,.$$
Hence for $\beta'_{j+1}$ we are in condition~(Z3b) again. We can choose 
$\beta'_{j+1}=\alpha_{j+1}$ and the remaining entries to be $0$.
Finally, 
  $$\comp(\alpha-\beta')=\comp(\alpha_1-\beta_1,\ldots,\alpha_{j-1}-\beta_{j-1},
  \alpha_j-\alpha_j+\alpha_{j+1},0,\alpha_{j+2},\cdots,\alpha_m)=\gamma\,.$$
This shows the claim in Equation~\eqref{eq:betaprime}. 

We remark that if $\alpha_j\ge\alpha_{j+1}$, 
then it is not possible to find $\beta'\in \Z_{s,\alpha}^\gamma$ such that $\beta'_p=\alpha_p$ for $p>j$. 
Indeed, using Equation~\eqref{eq:compbeta},
  \begin{align*}
     [\alpha_1-\beta_1,\ldots,\alpha_{j-1}-\beta_{j-1},\alpha_{j+1},\ldots,\alpha_m]&=\gamma=\comp(\alpha-\beta')\\
      &=[\alpha_1-\beta'_1,\ldots,\alpha_{p-1}-\beta'_{p-1},\alpha_{p+1},\ldots,\alpha_m]
   \end{align*}   
       and $p>j$ implies that $\beta_i=\beta'_i$ for $1\le i<j$ and $\alpha_{j+1}=\alpha_j-\beta'_j$. This gives
  $$s-\beta'_1-\cdots-\beta'_{j-1}=s-\beta_1-\cdots-\beta_{j-1}=\alpha_j$$
so $\beta'_j$ must satisfy condition~(Z3b). But we see above that 
$\beta'_j=\alpha_{j}-\alpha_{j+1}\ge 0$ a contradiction to condition~(Z3b).
This shows that using Equation~\eqref{eq:betaprime} repeatedly we can get vector  
$\beta'\in \Z_{s,\alpha}^\gamma$ such that $\beta'_r=\alpha_r$
for $j<r$ as long as $\alpha_j<\alpha_{j+1}<\ldots<\alpha_{r}$ but no further if $\alpha_r\ge\alpha_{r+1}$.

Let us go back to $\beta\in \Z_{s,\alpha}^\gamma$ and $j$ as in Equation~\eqref{eq:mybeta}. 
From the definition of $k$ in (A3), it is clear that $k\le j$. Also, using Equation~\eqref{eq:compbeta},
we have $\alpha_i=\gamma_{i-1}=\alpha_{i-1}-\beta_{i-1}$ for all $k<i\le j$. This gives us that
  $$\beta_{i-1}=\alpha_{i-1}-\alpha_i,\quad\hbox{for all }k<i\le j.$$
Given that $s-\beta_1-\ldots-\beta_{j-1}=\alpha_j$ we get that for all $k<i\le j$
  $$s-\beta_1-\beta_{i-1}=\alpha_j+\beta_{j-1}+\cdots+\beta_{i}=\alpha_j+\alpha_{j-1}-\alpha_j+
  \cdots+\alpha_i-\alpha_{i+1}=\alpha_i\,.$$
This gives us that $\beta_k,\beta_{k+1},\ldots,\beta_j$ must satisfy condition~(Z3b). 
Since none of $\beta_{i-1}=\alpha_{i-1}$ for $i\le j$ we must have
  $$\alpha_{i-1}-\alpha_i =\beta_{i-1} <0\,.$$
 Hence $\alpha_k<\alpha_{k+1}<\cdots<\alpha_j$. It is clear now that the vector 
 $(\alpha_1-\gamma_1,\ldots,\alpha_{k-1}-\gamma_{k-1},\alpha_k,0,\ldots,0)\in
 \Z_{s,\alpha}^\gamma$. 
  This shows that all the vectors to the right hand side of (A3) are contained in $\Z_{s,\alpha}^\gamma$ 
  and no vector $\beta'\in\Z_{s,\alpha}^\gamma$ will satisfy $\beta'_p=\alpha_p$ for $p>r$.
 
Condition (A3) will follow once we show that no $\beta'\in \Z_{s,\alpha}^\gamma$ has $\beta'_p=\alpha_p$ for $p<k$. 
But if such $p$ exists it would imply that
$\beta'_i=0$ for $i>p$. Then $\gamma_i=\alpha_{i+1}-\beta_{i+1}=\alpha_{i+1}$ for all $p<i$. 
This is a contradiction to the minimality of $k$ since $p<k$.
\end{proof}

We are now ready to state and prove our main theorem. Let
$$ c_{s,\alpha}^\gamma =  \sum_{\beta\in \Z_{s,\alpha}^\gamma}\sgn(\beta) \,.
$$
That is $c_{s,\alpha}^\gamma$ is the coefficient of $\fS_\gamma$ in Equation~\eqref{eq:FperpImm2}.
\begin{Theorem} \label{compoThm}
For a positive integer $s$ and a composition $\alpha$, 
\begin{equation} \label{formula}
F^\perp_s \fS_\alpha = \sum_{\gamma\models |\alpha|-s} c_{s,\alpha}^\gamma \fS_{\gamma}\,,
\end{equation}
where
\begin{equation*}
  \label{basic}
  c_{s,\alpha}^\gamma =
  \left\{
  \begin{array}{ll}
  \sgn(\alpha-\gamma) \, \,  & \mbox{if  \ $\ell(\gamma)=\ell(\alpha)$ and }\\
 & \mbox{$\alpha-\gamma$ satisfies conditions (Z1)--(Z3)},\\
   \\[3pt]
 \sgn(\alpha_1-\gamma_1,\ldots,\alpha_{k-1}-\gamma_{k-1}) \, \,   \, \,  
& \mbox{if  \ $\ell(\gamma)=\ell(\alpha)-1$, if $r-k$ is even, and}\\
& \mbox{$(\alpha_1-\gamma_1,\ldots,\alpha_{k-1}-\gamma_{k-1},\alpha_k,0\ldots,0) $}\\
& \mbox{satisfies conditions (Z1)--(Z3)}
   \\[3pt]
0 \, \,  & \mbox{otherwise}~.
 \end{array}
  \right. 
  \end{equation*}
  Here, $r,k$ depend on $\alpha$ and $\gamma$ as defined in Proposition~\ref{prop:coefset} (A3).
\end{Theorem}
\begin{proof}
This is a direct application of Proposition~\ref{prop:coefset}. 
Condition (A1) implies that $c_{s,\alpha}^\gamma=0$ if $\ell(\gamma)<\ell(\alpha)-1$ or 
$\ell(\gamma)>\ell(\alpha)$. In case of (A2), if non-empty, the set $\Z_{s,\alpha}^\gamma$ 
contains a unique vector $\alpha-\gamma$, hence $c_{s,\alpha}^\gamma=\sgn(\alpha-\gamma)$ when 
$\ell(\gamma)=\ell(\alpha)$. The case (A3) is slightly more interesting. Let $r,k$ be as in (A3). 
If non-empty, the set $\Z_{s,\alpha}^\gamma$ contains exactly $r-k+1$ vectors. 
The difference of negative components between two consecutive vectors of $\Z_{s,\alpha}^\gamma$ 
as listed in (A3) is exactly one. This shows that in this case
\begin{align*}
   c_{s,\alpha}^\gamma = \sgn(\alpha_1-\gamma_1,\ldots,\alpha_{k-1}-\gamma_{k-1},\alpha_k,0,\ldots,0)
   \big(1 -1 +1 -\cdots +(-1)^{r-k}\big)\,.
\end{align*}
The result follows depending on the parity of $r-k$.
\end{proof}

We conclude our study by considering the special case where $\alpha$ is a partition. 
In this case, all coefficients are positive as we will see below.

\begin{Lemma} \label{partiZs}
Let $\lambda$ be a partition, and $s$ a positive integer. 
\begin{itemize}\itemsep = 1mm
\item[\rm(i)] If $\lambda_1 < s$, then $\, \Z_{s,\lambda} = \emptyset$.
\item[\rm(ii)] If $\lambda_1 = s$, then $\, \Z_{s,\lambda} = \{(s, 0, \ldots, 0) \}$.
\item[\rm(iii)] If $\lambda_1 > s$, then $\, \sgn(\beta) = 1$ for all $\, \beta \in \Z_{s, \lambda}$.
\end{itemize}
\end{Lemma}

\begin{proof}
Let $\lambda$ be a partition of length $m$, $s$ a positive integer and $\beta \in \Z_{s, \lambda}$.
  
(i). From (Z3c), since $\lambda_1<s$, we must have that $\beta_1 < 0$. 
Using again (Z3c) for $i=2,\ldots,m$ successively, we have 
$\lambda_i \leq \lambda_1 < s < s - (\beta_1 + \ldots + \beta_{i-1})$, which yields that $\beta_i < 0$. 
Thus, all the entries in $\beta$ are negative integers. This is in opposition  to (Z1) that requires 
the sum  to be $s$. Therefore $\Z_{s,\lambda} = \emptyset$. 

(ii). Using (Z3b) gives that either $\beta = (s, 0, \ldots, 0)$ or $\beta_1 < 0$. 
Reasoning, as in (i), we conclude that $\beta_1 < 0$ is not possible. Hence $\Z_{s,\lambda} = \{(s,0,\ldots,0)\}$. 

(iii). If $\lambda_1 > s$, then (Z3a) applies to give $0\leq \beta_1 \leq s$. 
Next, we compare $\lambda_2$ with $s - \beta_1$.
If $\lambda_2 < s - \beta_1$, then  (Z3c) would imply $\beta_2<0$, which reasoning as before, 
yields that $\beta_i < 0$ for $i > 1$. In such case, $\beta_1+\cdots+\beta_m<\beta_1<s$ since 
$0\le\lambda_2<s-\beta_1$. This contradicts (Z1) and we cannot have $\lambda_2 < s - \beta_1$. 
When  $\lambda_2 = s - \beta_1$, the only possibility is $\beta_2 = s - \beta_1$, $\beta_3 = 
\ldots = \beta_m = 0$.
When $\lambda_2 > s - \beta_1$ we have  (Z3a) implies $0\leq \beta_2 \leq s - \beta_1$. 
Continuing successively  this procedure for $i=3,\ldots,m$, we show that all the entries in 
$\beta$ are non-negative integers, and therefore $\sgn(\beta) = 1$.  
\end{proof}

A direct consequence of Lemma~\ref{partiZs} is the following result
\begin{Corollary} 
Let $\lambda$ be a partition, and $s$ a positive integer. We have
\begin{equation*}
  \label{basic}
  F^\perp_s \fS_\lambda  =
  \left\{
  \begin{array}{ll}
   0 \, \, & \mbox{ if  \ $\lambda_1 < s$},
  \\[3pt]
   \fS_{\tail(\lambda)} \, \,  & \mbox{ if  \ $\lambda_1 = s$},
   \\[3pt]
{\displaystyle \sum_{\gamma\models |\lambda|-s} c_{s,\lambda}^\gamma \fS_{\gamma}} \, \,  & 
\mbox{ if  \ $\lambda_1 > s$},
 \end{array}
  \right. 
  \end{equation*}
 where $c_{s,\lambda}^\gamma=1$ or $0$.
\end{Corollary}

\begin{Example}
{\rm
For $\alpha = [3, 1, 2]$ and $s = 2$, an application of \eqref{formula} gives:
\[
F^\perp_2 \fS_{312} = \fS_{112},
\]
To see this, we consider the compositions $\gamma\models 6-2=4$ of length $2$ and $3$.
For length $3$, the only possible $\gamma$'s are $[1,1,2]$, $[1,2,1]$ and $[2,1,1]$. 
We consider the unique vector $\beta=\alpha-\gamma$
and check if it satisfies the conditions (Z1)--(Z3).
The three vectors $\beta$ that we get in this case are $(2,0,0)$, $(2,-1,1)$ and $(1,0,1)$. 
The vector $\beta=(2,-1,1)$ fails the condition (Z3a) in the second coordinate
since $\alpha_2=1 > 2-2=s-\beta_1$ but $\beta_2\not\ge 0$. 
 The vector $\beta=(1,0,1)$ fails the condition (Z3b) in the second coordinate
since $\alpha_2=1=2-1=s-\beta_1$ but $\beta_2\ne \alpha_2$ nor $\beta_2\not< 0$.
We have
 $$ c_{2, 312}^{112}=\sgn(2,0,0)=1;\quad c_{2, 312}^{121}=c_{2, 312}^{211}=0.$$
Now for compositions $\gamma$ of length 2 the possibilities are $[1,3]$, $[2,2]$ and $[3,1]$.
We compute $(k,r)$ as in (A3) for each $\gamma$ and get $(3,3)$, $(2,3)$, $(3,3)$, respectively.
Since $3-2=1$ is not even, we can disregard $\gamma=[2,2]$. For $\gamma=[1,3]$ we have $k=3$ and the vector
$\beta=(\alpha_1-\gamma_1,\alpha_2-\gamma_2,\alpha_3)=(2,-2,2)$ fails condition (Z3a) in the second 
coordinate (same argument as above).
For $\gamma=[3,1]$ we have again $k=3$ and the vector
$\beta=(\alpha_1-\gamma_1,\alpha_2-\gamma_2,\alpha_3)=(0,0,2)$ fails condition (Z3c) in the second coordinate since
 $\alpha_2=1<2-0=s-\beta_1$ but $\beta_2\not<0$. We have
  $$ c_{2, 312}^{13}= c_{2, 312}^{22}=c_{2, 312}^{31}=0.$$
}
\end{Example}

\end{document}